\documentclass[oneside,english]{amsart}
\usepackage[T1]{fontenc}
\usepackage[latin9]{inputenc}
\usepackage{amsthm}
\usepackage[all]{xy}

\makeatletter
\numberwithin{equation}{section}
\numberwithin{figure}{section}
\theoremstyle{plain}
\newtheorem{thm}{\protect\theoremname}
  \theoremstyle{definition}
  \newtheorem{defn}[thm]{\protect\definitionname}
  \theoremstyle{remark}
  \newtheorem{rem}[thm]{\protect\remarkname}

\makeatother

\usepackage{babel}
  \providecommand{\definitionname}{Definition}
  \providecommand{\remarkname}{Remark}
\providecommand{\theoremname}{Theorem}

\begin{document}

\title{A note on the metrizability of spaces}

\author{Ittay Weiss}
\begin{abstract}
With the blessing of hind sight we consider the problem of metrizability
and show that the classical Bing-Nagata-Smirnov Theorem and a more
recent result of Flagg give complementary answers to the metrization
problem, that are in a sense dual to each other.
\end{abstract}
\maketitle

\section{Introduction}

Consider the categories ${\bf Top}$ of topological spaces and continuous
mappings, and ${\bf Met}$ of metric spaces and continuous mappings.
The standard construction which associates with a metric space $(S,d)$
the topology $\mathcal{O}(S,d)$ generated by the open balls $\{B_{r}(x)\mid x\in S,r>0\}$,
where $B_{r}(x)=\{y\in S\mid d(x,y)<r\}$, is the object part of a
functor $\mathcal{O}:{\bf Met}\to{\bf Top}$ which on morphisms is
given by $f\mapsto f$. Of course, the functor $\mathcal{O}$ is not
invertible, not even in the weakest reasonable sense; the categories
${\bf Met}$ and ${\bf Top}$ are simply not equivalent. It is natural
though to consider ways in which $\mathcal{O}$ can be non-trivially
turned into an equivalence. Notice that since $\mathcal{O}$ is fully-faithful,
the problem is equivalent to asking for ways of turning $\mathcal{O}$
into an essentially surjective functor (without altering the morphisms
too much so as to retain the fully-faithfulness of $\mathcal{O}$).
There are two immediate candidates for approaches to this problem,
namely either by restricting the codomain of $\mathcal{O}$ or by
enlarging the domain. We state this as the following \emph{metrization
problem}: Augment the functor $\mathcal{O}:{\bf Met}\to{\bf Top}$
to obtain a diagram
\[
\xymatrix{{\bf Top_{M}}\ar[r]\ar@<2pt>[rd]^{M} & {\bf Top}\ar@<2pt>[dr]^{M}\\
 & {\bf Met}\ar[r]\ar[u]^{\mathcal{O}}\ar@<2pt>[ul]^{\mathcal{O}} & {\bf Met_{T}}\ar@<2pt>[ul]^{\mathcal{O}}
}
\]
such that
\begin{itemize}
\item Both horizontal arrows are inclusion functors.
\item Each of the two $(M,\mathcal{O})$ pairs is an equivalence of categories.
\item $\mathcal{O}:{\bf Met_{T}}\to{\bf Top}$ extends $\mathcal{O}:{\bf Met}\to{\bf Top}$
along ${\bf Met}\to{\bf Met_{T}}$.
\item $\mathcal{O}:{\bf Met}\to{\bf Top}$ extends $\mathcal{O}:{\bf Met}\to{\bf Top_{M}}$
along ${\bf Top_{M}}\to{\bf Top}$. 
\end{itemize}
One obvious way of restricting the codomain is to consider the image
of $\mathcal{O}$ in ${\bf Top}$ and trivially obtain that $\mathcal{O}$
is an equivalence onto its image. Topologically identifying this image
is of course nothing but the problem of specifying a necessary and
sufficient condition for a space to be metrizable. In the years 1950
and 1951, Bing, Nagata, and Smirnov proved, independently, the following
result. 
\begin{thm}
A topological space $S$ is metrizable if, and only if, it is $T_{0}$,
regular, and admits a $\sigma$-discrete basis. 
\end{thm}
Thus, if we restrict ${\bf Top}$ to the full subcategory ${\bf Top}_{M}$
spanned by the regular $T_{0}$ spaces that admit a $\sigma$-discrete
basis, then the functor $\mathcal{O}:{\bf Met}\to{\bf Top}$ factors
through the inclusion ${\bf Top_{M}}\to{\bf Top}$, giving rise to
$\mathcal{O}:{\bf Met}\to{\bf Top_{M}}$, clearly an equivalence.
In particular, this functor together with the functor $M:{\bf Top_{M}}\to{\bf Met}$
which on objects sends $S$ to $M(S)$ where $M(S)$ is any of the
metric spaces $(S,d)$ guaranteed to exist by the Bing-Nagata-Smirnov
Theorem (and $f\mapsto f$ on morphisms) fulfill the requirements
of the left triangle in the metrization problem above. 

We now turn to identify a solution to the other triangle in the metrization
problem. What is required now is an extension of the category ${\bf Met}$
such that for it the statement that every space is metrizable is correct.
The approach we present here is taken from Flagg's \cite{MR1452402}.

First, we recall some definitions regarding posets and lattices. In
a poset $V$, joins and meets are denoted by $\bigvee$ and $\bigwedge$
respectively. The \emph{well above }relation on $V$ is denoted by
$x\prec y$ (which is read as ``$y$ is well-above $x$'') if for
all subsets $S\subseteq V$, if $x\ge\bigwedge S$ then there exists
some $s_{0}\in S$ such that $y\ge s_{0}$. A \emph{complete lattice
}is a poset where every subset has a meet and a join. In particular,
a complete lattice has a smallest element $0$ and a largest element
$\infty$. A complete lattice is \emph{completely distributive }if
$y=\bigwedge\{x\in V\mid x\succ y\}$, and it is a \emph{value distributive
lattice }if moreover $\infty\succ0$ and $x\wedge y\succ0$ whenever
$x\succ0$ and $y\succ0$. A \emph{quantale} is a complete lattice
$Q$ together with an associative and commutative binary operation
$+$ on $Q$ such that $x+0=x$ holds for all $x\in Q$ and $x+\bigwedge S=\bigwedge x+S$
for all $x\in Q$ and $S\subseteq Q$. 
\begin{defn}[Flagg]
A \emph{value quantale }is a quantale $V$ such that as a complete
lattice it is a value distributive lattice. \end{defn}
\begin{rem}
In the literature, the more common definition of quantale requires
addition to commute with joins (and $\infty)$ rather than with meets
(and $0$). In other words, what Flagg calls a quantale will today
be more commonly referred to as an \emph{op-quantale}.\end{rem}
\begin{defn}
A $V$-\emph{continuity space}, where $V$ is a value quantale, is
a pair $(X,d)$ where $X$ is a set and $d:X\times X\to V$ is a function
satisfying $d(x,x)=0$ for all $x\in X$ and $d(x,z)\le d(x,y)+d(y,z)$
for all $x,y,z\in X$. $(X,d)$ is said to be \emph{separated }if
$d(x,y)=d(y,x)=0$ implies $x=y$, and $(X,d)$ is \emph{symmetric
}if $d(x,y)=d(y,x)$ holds for all $x,y\in X$.
\end{defn}
The prototypical value quantale is $[0,\infty]$ with its usual ordering
and with ordinary addition. A $[0,\infty]$-continuity space is then
precisely a semi-quasimetric space. The separated and symmetric $[0,\infty]$-continuity
spaces are precisely ordinary metric spaces (which are allowed to
attain $\infty$ as a distance). 

Let ${\bf Met_{T}}$ be the category whose objects are all pairs $(V,X)$
where $V$ is a value quantale and $X$ is a $V$-continuity space.
The morphisms $f:(V,X)\to(W,Y)$ in ${\bf Met_{T}}$ are functions
$f:X\to Y$ which are \emph{continuous }in the sense that for all
$x\in V$ and for all $\varepsilon\succ0$ in $W$ there exists a
$\delta\succ0$ in $V$ such that $d(f(x),f(y))\prec\varepsilon$
whenever $d(x,y)\prec\delta$. Evidently, ${\bf Met}$ embeds in ${\bf Met_{T}}$
by considering all separated and symmetric $V$-continuity spaces
for the value quantale $V=[0,\infty]$. 

Given any $V$-continuity space $X$, $x\in X$ and $\epsilon\succ0$
in $V$, th\emph{e set} $B_{\varepsilon}(x)=\{y\in X\mid d(x,y)\prec\varepsilon\}$
is called the \emph{open ball }of radius $\varepsilon$ about $x$.
Declare a set $U\subseteq X$ to be \emph{open }if for all $x\in U$
there exists $\varepsilon\succ0$ in $V$ with $B_{\varepsilon}(x)\subseteq U$. 
\begin{thm}
The collection of all open sets in a $V$-continuity space $X$ is
a topology \end{thm}
\begin{proof}
See the proof of Theorem 4.2 in \cite{MR1452402}.
\end{proof}
We thus obtain the functor $\mathcal{O}:{\bf Met_{T}}\to{\bf Top}$
sending a $V$-continuity space to the topological space generated
by its open sets and given on morphisms by $f\mapsto f$. The verification
that this is indeed a fully faithful functor is the claim that a function
$f:X\to Y$, for a $V$-continuity space $X$ and a $W$-continuity
space $Y$, is continuous in the $\epsilon-\delta$ sense above if,
and only if, it is continuous with respect to the induced open ball
topologies. The proof is identical to the analogous claim for ordinary
metric spaces. Clearly, $\mathcal{O}:{\bf Met_{T}}\to{\bf Top}$ extends
$\mathcal{O}:{\bf Met}\to{\bf Top}$ and it is obvious that $\mathcal{O}:{\bf Met_{T}}\to{\bf Top}$
is fully-faithful. It thus remains to show that it is essentially
surjective. In fact, it is surjective on objects as follows from Theorem
4.15 of \cite{MR1452402} which states the following:
\begin{thm}
Let $(S,\tau)$ be a topological space. There exists a value quantale
$V_{\tau}$ and a $V_{\tau}$-continuity space $(S,d)$ such that
$(S,\tau)=\mathcal{O}(S,d)$. 
\end{thm}
Inspecting the proof of that theorem gives rise to a very explicit
construction of the functor $M:{\bf Top}\to{\bf Met_{T}}$, namely
$M(S,\tau)=(V,S)$ where $V=\Omega(\tau)$ and $d:S\times S\to\Omega(\tau)$
is given by $d(x,y)=\{F\subseteq_{f}\tau\mid\forall U\in F,\, x\in U\implies y\in U\}$. 

This establishes Flagg's construction as a solution to the second
triangle in the statement of the metrization problem above and completes
our note on the complimentary nature of the Bing-Nagata-Smirnov metrization
theory and Flagg's metrization theorem. 

\bibliographystyle{plain}
\bibliography{/Users/ittayweiss/Documents/PapersInProgress/generalReferences}

\end{document}